\documentstyle[amscd,amssymb,xypic,verbatim,11pt]{amsart}
\xyoption{all}
\newcommand{\nc}{\newcommand}
\newcommand{\rnc}{\renewcommand}

\nc{\exto}[1]{\stackrel{#1}{\longrightarrow}}
\nc{\dlim}{{\mathop{\lim\limits_{\longrightarrow}}}}
\nc{\lan}{\big\langle}
\nc{\ran}{\big\rangle}

\nc{\kk}{{\mathsf{k}}}
\nc{\ix}{{\mathsf{i}}}
\nc{\jx}{{\mathsf{j}}}

\nc{\C}{{\mathbb{C}}}
\nc{\HH}{{\mathbb{H}}}
\nc{\LL}{{\mathbb{L}}}
\nc{\PP}{{\mathbb{P}}}
\nc{\RR}{{\mathbb{R}}}
\nc{\SS}{{\mathbb{S}}}
\nc{\TT}{{\mathbb{T}}}
\nc{\QQ}{{\mathbb{Q}}}
\nc{\ZZ}{{\mathbb{Z}}}

\nc{\CA}{{\mathcal{A}}}
\nc{\CB}{{\mathcal{B}}}
\nc{\CC}{{\mathcal{C}}}
\nc{\D}{{\mathcal{D}}}
\nc{\CE}{{\mathcal{E}}}
\nc{\CF}{{\mathcal{F}}}
\nc{\CG}{{\mathcal{G}}}
\nc{\CH}{{\mathcal{H}}}
\nc{\CJ}{{\mathcal{J}}}
\nc{\CL}{{\mathcal{L}}}
\nc{\CM}{{\mathcal{M}}}
\nc{\CN}{{\mathcal{N}}}
\nc{\CO}{{\mathcal{O}}}
\nc{\CQ}{{\mathcal{Q}}}
\nc{\CR}{{\mathcal{R}}}
\nc{\CS}{{\mathcal{S}}}
\nc{\CT}{{\mathcal{T}}}
\nc{\CU}{{\mathcal{U}}}
\nc{\CV}{{\mathcal{V}}}
\nc{\CW}{{\mathcal{W}}}
\nc{\CX}{{\mathcal{X}}}
\nc{\CY}{{\mathcal{Y}}}
\nc{\CZ}{{\mathcal{Z}}}
\nc{\CMo}{{\mathcal{M}^\circ}}
\nc{\Co}{{{C}^\circ}}

\nc{\BY}{{\overline{Y}}}
\nc{\BZ}{{\overline{Z}}}
\nc{\BYD}{{\overline{Y}{}^{|D|}}}
\nc{\OZ}{{\overline{Z}}}
\nc{\bg}{{\bar{g}}}

\nc{\bq}{{\mathbf{q}}}
\nc{\BD}{{\mathbf{D}}}
\nc{\BG}{{\mathbf{G}}}
\nc{\BL}{{\mathbf{L}}}
\nc{\BM}{{\mathbf{M}}}
\nc{\BP}{{\mathbf{P}}}
\nc{\BPr}{{\mathsf{P}}}
\nc{\BR}{{\mathbf{R}}}
\nc{\BRO}[1]{{{\mathbf{R}}^{\circ}_{#1}}}
\nc{\BRD}[1]{{{\mathbf{R}}^{|D|}_{#1}}}
\nc{\BRP}[1]{{{\mathbf{R}}^{1}_{#1}}}
\nc{\BRTP}[1]{{{\mathbf{\tilde{R}}}{}^{1}_{#1}}}
\nc{\BS}{{\mathbf{S}}}
\nc{\BT}{{\mathbf{T}}}
\nc{\BMS}{{{\mathbf{M}}^{{s}}}}
\nc{\BMSS}{{{\mathbf{M}}^{{ss}}}}
\nc{\BMZ}{{\mathbf{M}^{\circ}}}
\nc{\BCL}{{\mathbf{L}}}

\nc{\PCC}{{{}^\perp\CC}}

\nc{\Cl}{{\mathsf{Cliff}}}
\nc{\Clev}{{\mathop{\mathsf{Cliff}}^{\circ}}}
\nc{\FA}{{\mathfrak{A}}}
\nc{\FB}{{\mathfrak{B}}}
\nc{\FI}{{\mathfrak{I}}}
\nc{\FZ}{{\mathfrak{Z}}}

\nc{\TFA}{{\tilde{\mathfrak{A}}}}
\nc{\TFB}{{\tilde{\mathfrak{B}}}}

\nc{\fa}{{\mathfrak{a}}}
\nc{\fg}{{\mathfrak{g}}}
\nc{\fp}{{\mathfrak{p}}}
\nc{\FD}{{\mathfrak{D}}}
\nc{\FE}{{\mathfrak{E}}}
\nc{\FL}{{\mathfrak{L}}}
\nc{\FM}{{\mathfrak{M}}}
\nc{\FR}{{\mathfrak{R}}}
\nc{\FS}{{\mathsf{S}}}

\nc{\sfc}{{\mathsf{c}}}
\nc{\sfch}{{\mathsf{ch}}}
\nc{\sfh}{{\mathsf{h}}}

\nc{\SK}{{\mathsf{K}}}
\nc{\SO}{{\mathsf{O}}}
\nc{\SQ}{{\mathsf{Q}}}
\nc{\SPV}{{\mathsf{S}^+\mathsf{V}}}
\nc{\SMV}{{\mathsf{S}^-\mathsf{V}}}
\nc{\SPMV}{{\mathsf{S}^\pm\mathsf{V}}}
\nc{\SX}{{S_X}}
\nc{\SY}{{S_Y}}
\nc{\phipsi}{{q}}
\nc{\eps}{\varepsilon}

\nc{\pim}{{\pi_-}}
\nc{\pip}{{\pi_+}}

\nc{\BE}{{\overline{\CE}}}
\nc{\TE}{{\tilde{\CE}}}
\nc{\TQ}{{\tilde{Q}}}
\nc{\TCF}{{\tilde{\CF}}}
\nc{\TCG}{{\tilde{\CG}}}
\nc{\TCL}{{\tilde{\CL}}}
\nc{\TF}{{\tilde{F}}}
\nc{\TW}{{\tilde{W}}}
\nc{\TCB}{{\widetilde{\CB}}}
\nc{\TCC}{{\tilde{\CC}}}
\nc{\TCX}{{\tilde{\CX}}}
\nc{\TCY}{{\tilde{\CY}}}
\nc{\TPhi}{{\tilde{\Phi}}}
\nc{\OPhi}{{\bar{\Phi}}}
\nc{\txi}{{\tilde{\xi}}}
\nc{\tp}{{\tilde{p}}}
\nc{\tq}{{\tilde{q}}}
\nc{\tzeta}{{\tilde{\zeta}}}
\nc{\tpi}{{\tilde{\pi}}}
\nc{\Tsi}{{\tilde{\Sigma}}}

\nc{\HCB}{{\widehat{\CB}}}
\nc{\HE}{{\widehat{\CE}}}
\nc{\HS}{{\widehat{S}}}
\nc{\HX}{{\hat{X}}}
\nc{\hxi}{{\hat{\xi}}}

\nc{\UH}{{\mathcal{H}}}

\nc{\TM}{{\widetilde{M}}}
\nc{\TCM}{{\widetilde{\CM}}}
\nc{\TS}{{\widetilde{S}}}
\nc{\TU}{{\widetilde{U}}}
\nc{\TX}{{\widetilde{X}}}
\nc{\TY}{{\widetilde{Y}}}
\nc{\TYO}{{{\widetilde{Y}}^\circ}}
\nc{\barf}{{\bar{f}}}
\nc{\te}{{\tilde{e}}{}}
\nc{\tf}{{\tilde{f}}}
\nc{\tg}{{\tilde{g}}}
\nc{\ti}{{\tilde{\imath}}}
\nc{\tj}{{\tilde{\jmath}}}
\nc{\ty}{{\tilde{y}}}
\nc{\tphi}{{\tilde{\phi}}}
\nc{\hf}{{\hat{f}}}

\nc{\urho}{{\underline{\rho}}}

\nc{\LRA}{\Leftrightarrow}
\nc{\RA}{\Rightarrow}
\nc{\lotimes}{\mathbin{\mathop{\otimes}\limits^{\mathbb{L}}}}
\nc{\CEnd}{\mathop{\mathcal{E}\mathit{nd}}\nolimits}
\nc{\CExt}{\mathop{\mathcal{E}\mathit{xt}}\nolimits}
\nc{\CHom}{\mathop{\mathcal{H}\mathit{om}}\nolimits}
\nc{\RH}{\mathop{{\mathsf{R}}\Gamma}\nolimits}
\nc{\RGamma}{\mathop{{\mathsf{R}}\Gamma}\nolimits}
\nc{\Cone}{\mathop{\mathsf{Cone}}\nolimits}
\nc{\RHom}{\mathop{\mathsf{RHom}}\nolimits}
\nc{\RCHom}{\mathop{\mathsf{R}\mathcal{H}\mathit{om}}\nolimits}
\nc{\RG}{\mathop{\mathsf{R\Gamma}}\nolimits}
\nc{\Hom}{\mathop{\mathsf{Hom}}\nolimits}
\nc{\Ext}{\mathop{\mathsf{Ext}}\nolimits}
\nc{\End}{\mathop{\mathsf{End}}\nolimits}
\nc{\Tor}{\mathop{\mathsf{Tor}}\nolimits}
\nc{\Tordim}{\mathop{\mathsf{Tor}\text{\rm-}\mathsf{dim}}\nolimits}
\nc{\Hilb}{\mathop{\mathsf{Hilb}}\nolimits}
\nc{\Spec}{\mathop{\mathsf{Spec}}\nolimits}
\nc{\Proj}{\mathop{\mathsf{Proj}}\nolimits}
\nc{\Pic}{\mathop{\mathsf{Pic}}\nolimits}

\nc{\Tw}{\mathop{\mathsf{Tw}}\nolimits}
\nc{\Ker}{\mathop{\mathsf{Ker}}\nolimits}
\nc{\Coker}{\mathop{\mathsf{Coker}}\nolimits}
\nc{\codim}{\mathop{\mathsf{codim}}\nolimits}
\nc{\sing}{{\mathsf{sing}}}
\nc{\supp}{\mathop{\mathsf{supp}}}
\nc{\perf}{{\mathsf{perf}}}
\nc{\rank}{\mathop{\mathsf{rank}}}
\nc{\Pf}{{\mathsf{Pf}}}
\nc{\Gr}{{\mathsf{Gr}}}
\nc{\OGr}{{\mathsf{OGr}}}
\nc{\Flag}{{\mathsf{Fl}}}
\nc{\Kosz}{{\mathsf{Kosz}}}
\nc{\LGr}{{\mathsf{LGr}}}
\nc{\GTGr}{{\mathsf{G_2Gr}}}
\nc{\GTF}{{\mathsf{G_2F}}}
\nc{\OF}{{\mathsf{OF}}}
\nc{\Fl}{{\mathsf{Fl}}}
\nc{\Bl}{{\mathsf{Bl}}}
\nc{\GL}{{\mathsf{GL}}}
\nc{\PGL}{{\mathsf{PGL}}}
\nc{\SL}{{\mathsf{SL}}}
\nc{\SP}{{\mathsf{Sp}}}
\nc{\Spin}{{\mathsf{Spin}}}
\nc{\Tot}{{\mathsf{Tot}}}
\nc{\ev}{{\mathsf{ev}}}
\nc{\od}{{\mathsf{odd}}}
\nc{\coev}{{\mathsf{coev}}}
\nc{\id}{{\mathsf{id}}}
\nc{\opp}{{\mathsf{opp}}}
\nc{\PS}{{{\PP^3}}}
\nc{\Qu}{{{Q^3}}}
\nc{\tdim}{\mathop{\Tor\dim}}
\nc{\ecart}{{\fbox{$\scriptstyle\mathsf{EC}$}}}
\nc{\ad}{{\mathop{\mathsf ad}}}
\nc{\gr}{{\mathop{\mathsf gr}}}
\nc{\qgr}{{\mathop{\mathsf qgr}}}
\nc{\tor}{{\mathop{\mathsf tor}}}
\rnc{\mod}{{\mathop{\mathsf mod}}}
\nc{\Mod}{{\mathop{\mathsf Mod}}}
\nc{\Coh}{{\mathop{\mathsf Coh}}}
\nc{\Ab}{{\mathop{\mathcal{A}\mathit{b}}}}
\nc{\QCoh}{{\mathop{\mathsf QCoh}}}

\nc{\AAV}{{\mathcal{AAV}}}

\nc{\Rep}{{\mathsf{Rep}}}

\nc{\Cubics}{{{\mathcal{S}}_3}}
\nc{\VFT}{{{\mathcal{S}}_{14}}}
\nc{\VFTE}{{{\mathcal{N}}_{\mathrm{reg,sm}}}}
\nc{\MX}{{\CM_X}}
\nc{\MY}{{\CM_Y}}
\nc{\MYE}{{\CM_{Y,\CE}}}
\nc{\Yd}{{Y_d}}
\nc{\Yfive}{{Y_5}}
\nc{\Xg}{{X_{2g-2}}}
\nc{\Xtt}{{X_{22}}}
\nc{\Xst}{{X_{16}}}
\nc{\Xtw}{{X_{12}}}
\nc{\Xe}{{X_{8}}}
\nc{\Xf}{{X_{4}}}

\nc{\git}{{/\!\!/\!{}_\chi}}

\theoremstyle{plain}

\newtheorem{theorem}{Theorem}[section]
\newtheorem{conjecture}[theorem]{Conjecture}
\newtheorem{lemma}[theorem]{Lemma}
\newtheorem{proposition}[theorem]{Proposition}
\newtheorem{corollary}[theorem]{Corollary}

\theoremstyle{definition}

\newtheorem{definition}[theorem]{Definition}

\theoremstyle{remark}

\newtheorem{remark}[theorem]{Remark}

\newenvironment{proof}{\noindent{\sf Proof:}}{\qed\medskip}

\title%[Scheme of lines on a family of quadrics: geometry and derived category]%
{On nodal Enriques surfaces and quartic double solids}
\author{Colin Ingalls}
\address{\sloppy
\parbox{0.9\textwidth}{
{\bf C.I.:\ }Department of Mathematics and Statistics,
University of New Brunswick,
\hfill\\[5pt]
Fredericton, NB E3B 5A3
Canada}\bigskip}
\email{cingalls@@unb.ca}
\thanks{C.I. was partially supported by a NSERC Discovery Grant.}
\author{Alexander Kuznetsov}
\address{\sloppy
\parbox{0.9\textwidth}{
{\bf A.K.:\ }Algebra Section, Steklov Mathematical Institute,
8 Gubkin str., Moscow 119991 Russia
\hfill\\[5pt]
The Poncelet Laboratory, Independent University of Moscow
\hfill
}\bigskip}
\email{akuznet@@mi.ras.ru}
\thanks{A.K. was partially supported by
RFFI grants
08-01-00297,
09-01-12170,
10-01-93110,
10-01-93113,
NSh-4713.2010.1.}
\date{}

\begin{document}

\maketitle

\begin{abstract}
We consider the class of singular double coverings $X \to \PP^3$ ramified in the degeneration locus $D$
of a family of 2-dimensional quadrics. These are precisely the quartic double solids constructed by Artin and Mumford
as examples of unirational but nonrational conic bundles. With such quartic surface $D$ one can associate
an Enriques surface $S$ which is the factor of the blowup of $D$ by a natural involution acting without fixed
points (such Enriques surfaces are known as nodal Enriques surfaces or Reye congruences). We show that
the nontrivial part of the derived category of coherent sheaves on this Enriques surface $S$ is equivalent
to the nontrivial part of the derived category of a minimal resolution of singularities of $X$.
\end{abstract}

\section{Introduction}

%The goal of this paper is to describe an interesting relation
%of the derived categories of nodal Enriques surfaces and special quartic double solids.

Recall that an Enriques surfaces $S$ is a smooth projective surface which can be represented
as the quotient of a K3 surface by an involution acting without fixed points.
Consecutively the canonical class $K_S$ of $S$ is a 2-torsion, $2K_S = 0$.
More than 10 years ago S.Zube~\cite{Zu} made the following nice observation

\begin{theorem}[\cite{Zu}]\label{zube}
Let $F_i^+$, $F_i^-$, $i =1,\dots,10$, be the multiple fibers of the $10$ elliptic pencils on $S$.
Then the line bundles $\{\CO_S(-F_i^+)\}_{i=1}^{10}$ form a completely orthogonal exceptional collection in $\D^b(S)$.
\end{theorem}

%there is an exceptional collection of length 10 in the bounded derived category
%$\D^b(S)$ of coherent sheaves on $S$. Indeed, the surface $S$ has 10 elliptic pencils,
%each of them has two multiple fibers, say $F_i^+$ and $F_i^-$, $i=1,\dots,10$.
%Then the line bundles $\CO_S(-F_i^+)$ form an exceptional collection.
Since then it was an intriguing question to describe the orthogonal subcategory
\begin{equation}\label{cas}
\CA_S = {}^\perp\langle \{ \CO_S(-F_i^+) \}_{i=1}^{10} \rangle.
\end{equation}
The answer to this question is still not known. Our goal is to give a partial
answer for {\sf nodal Enriques surfaces}. Recall that Enriques surface $S$ is called {\sf nodal}
(another name is {\sf Reye congruence}), if the image of $S$ with respect to a morphism given
by a linear system $\frac13\sum F_i^+$ (called {\sf Fano model of $S$}) is contained in a quadric.
The main result of the present paper is a description of the category $\CA_S$
for nodal Enriques surfaces in other terms. Actually, we show that $\CA_S$ can be represented as a semiorthogonal
component of the derived category of a very special quartic double solid.

Recall that a quartic double solid is a double covering of $\PP^3$ ramified in a quartic surface $D \subset \PP^3$.
The bridge between quartic double solids and nodal Enriques surfaces is provided by special quartic surfaces,
the so-called {\sf quartic symmetroids}.

A quartic symmetroid is a quartic surface in $\PP^3 = \PP(W)$ which is a discriminant
locus of a family of quadrics in another $\PP^3 = \PP(V)$ parameterized by $\PP(W)$,
such that the quadrics in the family have no common points in $\PP(V)$.
The quartic symmetroid $D \subset \PP(W)$ parameterizes degenerate quadrics
in the family. It has 10 singular points corresponding to the quadrics of corank 2.
It turns out that the resolution of singularities $D'$ of $D$ has a natural
involution $\iota$ without fixed points and the quotient $S := D'/\iota$ is a nodal
Enriques surface.

On the other hand, one can consider the double covering $X \to \PP(W)$ ramified in the symmetroid~$D$.
As it was shown by F.Cossec~\cite{Co} such double coverings are exactly the Artin--Mumford conic bundles~\cite{AM}.
The main result of the present paper is a representation of the subcategory $\CA_S \subset \D^b(S)$
as a semiorthogonal component in $\D^b(X^+)$, the derived category of a small resolution $X^+$ of singularities of $X$.
More precisely we show that there is a semiorthogonal decomposition
$$
\D^b(X^+) = \langle \CA_S, \CO_{X^+}(-h), \CO_{X^+} \rangle,
$$
where $h$ stands for the positive generator of the Picard group of $\PP(W)$.

There is a subtle point in the above result. Namely, in general a minimal resolution of singularities
$X^+ \to X$ is not algebraic, in fact it is a Moishezon variety. However, since each Moishezon variety
is an algebraic space, its derived category of coherent sheaves (which is by definition is the subcategory
of the unbounded derived category of $\CO$-modules with bounded and coherent cohomology) is well-defined.
Moreover, it has all the properties the usual derived categories of coherent sheaves have.

On the other hand, if one wants to avoid nonalgebraic varieties, one can consider the blowup $X'$ of $X$
which is again a resolution of singularities, not small. Then we conjecture that one can realize
the whole derived category $\D^b(S)$ of the Enriques surface $S$ inside $\D^b(X')$.
More precisely, we conjecture that there is a semiorthogonal decomposition
$$
\D^b(X') = \langle \D^b(S), \CO_{X'}(-h), \{ \CO_{X'}(-e_i) \}_{i=1}^{10}, \CO_{X'} \rangle,
$$
where $e_i$ are the classes of the exceptional divisors of the blowup $X' \to X$.
We are going to return to this conjecture in future.

The paper is organized as follows.
In Section~\ref{s-pre} we remind some notions and constructions, such as
semiorthogonal decompositions, mutation functors and Serre functors.
In Section~\ref{s-qds} we discuss the derived categories of quartic double solids.
In Section~\ref{s-symm} we recall the quartic symmetroids and state the main result of the paper.
In Section~\ref{s-proof} we give the proof of the main result.

{\bf Acknowledgement.}
The first author would like to thank A.Bondal for helpful discussions.
The second author would like to thank L.Katzarkov and D.Orlov for helpful discussions
and is very grateful to I.Dolgachev for sharing many interesting facts about Enriques surfaces.

\section{Preliminaries}\label{s-pre}

The base field is the field $\C$ of complex numbers.

If $X$ is a Moishezon variety we denote by $\D^b(X)$ the full subcategory
of the unbounded derived category of $\CO_X$-modules with bounded and coherent cohomology.
This category is triangulated. For any morphism $f:X \to Y$ of Moishezon varieties we denote
by $f_*:\D^b(X) \to \D^b(Y)$ and by $f^*:\D^b(Y) \to \D^b(X)$ the {\em derived}\/
push-forward and pull-back functors (in first case we need $f$ to be proper,
and in the second to have finite $\Tor$-dimension, these assumptions ensure the functors
to preserve both boundedness and coherence). Similarly, $\otimes$
stands for the {\em derived}\/ tensor product. If $F \in \D^b(X)$
the derived tensor product functor $F\otimes -:\D^b(X) \to \D^b(X)$ will be denoted by $\TT_F$.

For a proper morphism of finite $\Tor$-dimension $f:X \to Y$ we will also use the right adjoint $f^!$
of the push-forward functor, which is given by the formula
$$
f^!(F) \cong f^*(F)\otimes\omega_{X/Y}[\dim X - \dim Y],
$$
where $\omega_{X/Y}$ is the relative canonical line bundle.

\subsection{Semiorthogonal decompositions}

Let $\CT$ be a triangulated category.

\begin{definition}[\cite{BK,BO95}]
A {\sf semiorthogonal decomposition}\/ of a triangulated category $\CT$ is a sequence of
full triangulated subcategories $\CA_1,\dots,\CA_n$ in $\CT$ such that
$\Hom_{\CT}(\CA_i,\CA_j) = 0$ for $i > j$
and for every object $T \in \CT$ there exists a chain of morphisms
$0 = T_n \to T_{n-1} \to \dots \to T_1 \to T_0 = T$ such that
the cone of the morphism $T_k \to T_{k-1}$ is contained in $\CA_k$
for each $k=1,2,\dots,n$.
\end{definition}

We will write $\CT = \langle \CA_1,\CA_2,\dots,\CA_n \rangle$ for a semiorthogonal
decomposition of a triangulated category $\CT$ with components $\CA_1,\CA_2,\dots,\CA_n$.

An important property of a triangulated subcategory $\CA \subset \CT$ ensuring that
it can be extended to a semiorthogonal decomposition is admissibility.

\begin{definition}[\cite{BK,B}]
A full triangulated subcategory $\CA$ of a triangulated category $\CT$ is called
{\sf admissible}\/ if for the inclusion functor $i:\CA \to \CT$ there is
a right adjoint $i^!:\CT \to \CA$, and a left adjoint $i^*:\CT \to \CA$ functors.
\end{definition}

\begin{lemma}[\cite{BK,B}]\label{sos_sod}
$(i)$ If\/ $\CA_1,\dots,\CA_n$ is a semiorthogonal sequence
of admissible subcategories in a triangulated category $\CT$\/
{\rm(}i.e. $\Hom_{\CT}(\CA_i,\CA_j) = 0$ for $i > j${\rm)}\/ then
$$
\lan\CA_1,\dots,\CA_k,
{}^\perp\lan\CA_1,\dots,\CA_k\ran \cap \lan\CA_{k+1},\dots,\CA_n\ran{}^\perp,
\CA_{k+1},\dots,\CA_n\ran
$$
is a semiorthogonal decomposition.

$(ii)$ If\/ $\D^b(X) = \lan \CA_1,\CA_2,\dots,\CA_n \ran$ is a semiorthogonal decomposition
of the derived category of a smooth projective variety $X$ then each subcategory $\CA_i \subset \D^b(X)$
is admissible.
\end{lemma}

Actually the second part of the Lemma holds for any {\em saturated}\/ (see~\cite{BK}) triangulated category.

\begin{definition}[\cite{B}]
An object $F \in \CT$ is called {\em exceptional}\/ if $\Hom(F,F)=\C$
and $\Ext^p(F,F)=0$ for all $p\ne 0$. A collection of exceptional
objects $(F_1,\dots,F_m)$ is called {\em exceptional}\/ if
$\Ext^p(F_l,F_k)=0$ for all $l > k$ and all $p\in\ZZ$.
\end{definition}

Assume that $\CT$ is {\em $\Ext$-finite}\/ (which means that for any objects $G,G' \in \CT$
the graded vector space $\Ext^\bullet(G,G') := \oplus_{t\in\ZZ} \Hom(G,G'[t])$ is finite-dimensional).

\begin{lemma}[\cite{B}]\label{eo}
The subcategory $\lan F \ran$ of $\D^b(X)$ generated by an exceptional object $F$ is
admissible and is equivalent to the derived category of vector spaces $\D^b(\C)$.
\end{lemma}

As a consequence of~\ref{sos_sod} and of~\ref{eo} one obtains the following

\begin{corollary}[\cite{BO95}]\label{sodgen}
If $\CT$ is an $\Ext$-finite triangulated category then any exceptional collection $F_1,\dots,F_m$ in $\CT$
induces a semiorthogonal decomposition
$$
\CT = \langle \CA , F_1, \dots, F_m \rangle
$$
where $\CA = \langle F_1, \dots, F_m \rangle^\perp = \{F \in \CT\ |\ \text{$\Ext^\bullet(F_k,F) = 0$ for all $1 \le k \le m$}\}$.
%and all the other components are the subcategories of $\D^b(X)$ generated by $F_k$ {\rm(}each of these is equivalent to
%$\D^b(\C)$, the derived category of $\C$-vector spaces{\rm)}.
\end{corollary}

\subsection{Serre functors}

\begin{definition}[\cite{BK},\cite{BO01}]
Let $\CT$ be a triangulated category.
A covariant additive functor $\SS:\CT\to\CT$ is a {\sf Serre functor}\/
if it is a category equivalence and for all objects $F,G\in\CT$ there are
given bi-functorial isomorphisms $\Hom(F,G)\to\Hom(G,\SS(F))^\vee$.
\end{definition}

If a Serre functor exists then it is unique up to a canonical
functorial isomorphism. If $X$ is a smooth projective variety
then $\SS(F):=F\otimes\omega_X[\dim X]$ is a Serre functor
in $\D^b(X)$.

\begin{lemma}[\cite{B}]\label{mut_funct}
If $\CT$ admits a Serre functor $\SS$ and $\CA\subset\CT$ is right admissible
then $\CA$ admits a Serre functor $\SS_\CA = i^!\circ\SS\circ i$,
where $i:\CA \to \CT$ is the inclusion functor.
\end{lemma}
\begin{proof}
If $A,A'\in\CA$ then
$\Hom_\CA(A,i^!\SS i A') \cong
\Hom_\CT(i A,\SS i A') \cong
\Hom_\CT(i A',i A)^\vee \cong
\Hom_\CA(A',A)^\vee$.
\end{proof}

\subsection{Mutations}

If a triangulated category $\CT$ has a semiorthogonal decomposition then
usually it has quite a lot of them. More precisely, there are two groups
acting on the set of semiorthogonal decompositions --- the group of
autoequivalences of $\CT$, and a certain braid group. The action of the braid
group is given by the so-called mutations.

Roughly speaking, the mutated decomposition is obtained by dropping one of the components
of the decomposition and then extending the obtained semiorthogonal collection
by inserting new component at some other place as in Lemma~\ref{sos_sod}.
More precisely, the basic two operations are defined as follows.

\begin{lemma}[\cite{B}]\label{mutfun}
Assume that $\CA \subset \CT$ is an admissible subcategory, so that we have
two semiorthogonal decompositions $\CT = \lan \CA^\perp, \CA \ran$
and $\CT = \lan \CA, {}^\perp\CA \ran$.
Then there are functors $\LL_\CA,\RR_\CA: \CT \to \CT$
vanishing on~$\CA$ and inducing mutually inverse equivalences
${}^\perp\CA \to \CA^\perp$ and $\CA^\perp \to {}^\perp\CA$ respectively.
\end{lemma}
\begin{proof}
Let $i:\CA \to \CT$ be the embedding functor. For any $F \in \CT$ we define
$$
\LL_\CA(F) = \Cone (i i^! F \to F),
\qquad
\RR_\CA(F) = \Cone (F \to i i^* F)[-1].
$$
Note that the cones in this triangles are functorial due to the semiorthogonality.
All the properties are verified directly.
\end{proof}

The functors $\LL_\CA$ and $\RR_\CA$ are known as the {\sf left}\/ and the {\sf right mutation functors}.

\begin{remark}
If $\CA$ is generated by an exceptional object $E$ we can use explicit formulas for
the adjoint functors $i^!$, $i^*$ of the embedding functor $i:\CA \to \CT$. Thus we obtain
the following distinguished triangles
\begin{equation}\label{excmut}
\Ext^\bullet(E,F)\otimes E \to F \to \LL_E(F),
\qquad
\RR_E(F) \to F \to \Ext^\bullet(F,E)^\vee\otimes E.
\end{equation}
\end{remark}

It is easy to deduce from Lemma~\ref{mutfun} the following

\begin{corollary}[\cite{B}]
Assume that $\CT = \lan \CA_1,\CA_2,\dots,\CA_n \ran$ is a semiorthogonal decomposition with
all components being admissible. Then for each $1 \le k \le n-1$ there is a semiorthogonal decomposition
$$
\CT = \lan \CA_1,\dots,\CA_{k-1},
\LL_{\CA_k}(\CA_{k+1}),
%{}^\perp\lan\CA_1,\dots,\CA_{k-1}\ran \cap \lan\CA_{k},\CA_{k+2},\dots,\CA_n\ran{}^\perp,
\CA_k,\CA_{k+2},\dots,\CA_n \ran
$$
and for each $2 \le k \le n$ there is a semiorthogonal decomposition
$$
\CT = \lan \CA_1,\dots,\CA_{k-2},\CA_k,
\RR_{\CA_k}(\CA_{k-1}),
%{}^\perp\lan\CA_1,\dots,\CA_{k-2},\CA_k\ran \cap \lan\CA_{k+1},\dots,\CA_n\ran{}^\perp,
\CA_{k+1},\dots,\CA_n \ran
$$
\end{corollary}

There are two cases when the action of the mutation functors is particularly simple.

\begin{lemma}\label{perpmut}
Assume that $\CT = \lan \CA_1,\CA_2,\dots,\CA_n \ran$ is a semiorthogonal decomposition with
all components being admissible. Assume also that the components $\CA_k$ and $\CA_{k+1}$
are completely orthogonal, i.e. $\Hom(\CA_k,\CA_{k+1}) = 0$ as well as $\Hom(\CA_{k+1},\CA_k) = 0$. Then
$$
\LL_{\CA_k}(\CA_{k+1}) = \CA_{k+1}
\qquad\text{and}\qquad
\RR_{\CA_{k+1}}(\CA_k) = \CA_k,
$$
so that both the left mutation of $\CA_{k+1}$ through $\CA_k$ and the right mutation of $\CA_k$ through $\CA_{k+1}$
boil down to just a permutation and
$$
\CT = \lan \CA_1,\dots,\CA_{k-1},\CA_{k+1},\CA_k,\CA_{k+2},\dots,\CA_n \ran
$$
is the resulting semiorthogonal decomposition of $\CT$.
\end{lemma}

\begin{lemma}\label{longmut}
Let $X$ be a smooth projective algebraic variety and $\CA \subset \D^b(X)$ an admissible subcategory.
Then
$$
\LL_{\CA^\perp}(\CA) = \CA\otimes\omega_X
\qquad\text{and}\qquad
\RR_{{}^\perp\CA}(\CA) = \CA\otimes\omega_X^{-1}.
$$
\end{lemma}

An analogue of this Lemma holds for any triangulated category which has a Serre functor (see~\cite{BK}).
In this case tensoring by the canonical class should be replaced by the action of the Serre functor.
Whenever such mutation is performed we will say that
the component $\CA$ of a semiorthogonal decomposition $\D^b(X) = \langle \CA^\perp,\CA \rangle$
is {\sf translated to the left}
(respectively the component $\CA$ of a semiorthogonal decomposition $\D^b(X) = \langle \CA,{}^\perp\CA \rangle$
is {\sf translated to the right}).

We will also need the following evident observation.

\begin{lemma}\label{tensmut}
Let $\Phi$ be an autoequivalence of $\CT$. Then
$$
\Phi\circ\RR_\CA \cong \RR_{\Phi(\CA)}\circ\Phi,
\qquad
\Phi\circ\LL_\CA \cong \LL_{\Phi(\CA)}\circ\Phi.
$$
In particular, if $L$ is a line bundle on $X$
and $E$ is an exceptional object in $\D^b(X)$ then
$$
\TT_L\circ\RR_E \cong \RR_{E\otimes L}\circ\TT_L,
\qquad
\TT_L\circ\LL_E \cong \LL_{E\otimes L}\circ\TT_L,
$$
where $\TT_L:\D^b(X) \to \D^b(X)$ is the functor of tensor product by $L$.
\end{lemma}

\section{Derived category of a quartic double solid}\label{s-qds}

In this section we discuss the structure of the derived categories of quartic double solids
(smooth and nodal) and their resolutions of singularities.

\subsection{Smooth double covering}

We start with the following general result.

\begin{lemma}\label{pbec}
Let $f:X \to Y$ be a $2$-fold covering ramified in a divisor $D \subset Y$.
Assume that for some $H \in \Pic Y$ we have $D \sim 2H \sim -K_Y$.
Assume that $\CE_1,\dots,\CE_n$ is an exceptional collection in $\D^b(Y)$ such that
the collection
\begin{equation}\label{col1}
\CE_1(-H),\dots,\CE_n(-H),\CE_1,\dots,\CE_n
\end{equation}
is also exceptional. Then the collection
$$
f^*\CE_1,\dots,f^*\CE_n \in \D^b(X)
$$
is exceptional in $\D^b(X)$.
\end{lemma}
\begin{proof}
Note that $X = \Spec_Y(\CO_Y \oplus \CO_Y(-H))$, hence $f_*\CO_X \cong \CO_Y \oplus \CO_Y(-H)$.
Now we have
$$
\Ext^\bullet(f^*\CE_i,f^*\CE_j) =
\Ext^\bullet(\CE_i,f_*f^*\CE_j) =
\Ext^\bullet(\CE_i,\CE_j\otimes (\CO_Y \oplus \CO_Y(-H))) =
\Ext^\bullet(\CE_i,\CE_j) \oplus \Ext^\bullet(\CE_i,\CE_j(-H))).
$$
Note that the second summand vanishes for all $i,j$ and the first summand vanishes
for $i > j$ since the collection~\eqref{col1} is exceptional.
Moreover, for $i = j$ we obtain exceptionality of $f^*\CE_i$.
\end{proof}

Now once we have constructed an exceptional collection in $\D^b(X)$, we can extend
it to a semiorthogonal decomposition. It turns out that the additional component has
a very interesting property.

\begin{proposition}\label{serre}
Assume that the conditions of Lemma~{\rm\ref{pbec}} are satisfied.
Denote the orthogonal to the exceptional collection of the Lemma by
$\CA_X = \langle f^*\CE_1, \dots, f^*\CE_n \rangle^\perp$.
Then
$$
\D^b(X) = \langle \CA_X, f^*\CE_1, \dots, f^*\CE_n \rangle
$$
is a semiorthogonal decomposition. Moreover,
if collection~\eqref{col1} is full in $\D^b(Y)$ then
the Serre functor $\SS_{\CA_X}$ of the category $\CA_X$ is given by
$$
\SS_{\CA_X} \cong \tau[\dim X - 1],
$$
where $\tau$ is the involution of the double covering $f:X \to Y$.
\end{proposition}
\begin{proof}
The first part follows from Corollary~\ref{sodgen} and Lemma~\ref{pbec}.
So, the only thing we have to compute is the Serre functor. For this
we have to check that for any $\CF_1,\CF_2 \in \CA_X$ we have
$$
\Hom(\CF_1,\tau\CF_2[\dim X - 1])^\vee \cong \Hom(\CF_2,\CF_1),
$$
a bifunctorial isomorphism.
For this we start with the Serre duality on $X$. Note that $K_X = -H$, hence
$$
\Hom(\CF_1,\tau\CF_2[\dim X - 1])^\vee \cong \Hom(\tau\CF_2,\CF_1(-H)[1]) \cong \Hom(\CF_2,\tau\CF_1(-H)[1]).
$$
On the other hand, it is easy to see that for each $\CF_1$ we have a canonical distinguished triangle
$$
f^*f_*\CF_1 \to \CF_1 \to \tau\CF_1(-H)[1].
$$
So, to prove the claim it suffices to check that $\Hom(\CF_2,f^*f_*\CF_1) = 0$ for all $\CF_1,\CF_2 \in \CA_X$.
But since $K_{X/Y} = \CO(H)$ we have $f^! \cong f^*\circ\TT_{\CO(H)}$, so
$$
\Hom(\CF_2,f^*f_*\CF_1) \cong
\Hom(\CF_2(H),f^!f_*\CF_1) \cong
\Hom(f_*\CF_2(H),f_*\CF_1).
$$
By definition of $\CA_X$ we have $\Hom(\CE_i,f_*\CF_j) = 0$ for all $i,j$, hence
$$
f_*\CF_1 \in \langle \CE_1,\dots,\CE_n \rangle^\perp = \langle \CE_1(-H),\dots,\CE_n(-H) \rangle
$$
because~\eqref{col1} is full. By the same reason $f_*\CF_2(H) \in \langle \CE_1,\dots,\CE_n \rangle$.
The latter two subcategories are semiorthogonal, hence the $\Hom$ in question vanishes.
\end{proof}

\begin{remark}
This behavior of the Serre functor of $\CA_X$ has a nice generalization, see~\cite{K09}.
One can take $X$ to be a double covering of (or a hypersurface in) arbitrary smooth projective variety $Y$
which has a rectangular Lefschetz decomposition. Then analogously defined subcategory $\CA_X$ of $\D^b(X)$
is a fractional Calabi--Yau category.
\end{remark}

Now we apply this result to a quartic double solid. Recall that a quartic double solid
is a double covering of $\PP^3$ ramified in a quartic. Denote by $h$ the positive generator
of $\Pic\PP^3$ as well as its pullback to the double covering. The standard exceptional collection
$\CO(-3h),\CO(-2h),\CO(-h),\CO$ on $\PP^3$ has the form~\eqref{col1} for $H = 2h$, hence we obtain the following

\begin{corollary}
Let $D \subset \PP^3$ be a quartic surface and $f:X \to \PP^3$ --- the associated double covering.
Then $(\CO_X(-h),\CO_X)$ is an exceptional pair and we have a semiorthogonal decomposition
$$
\D^b(X) = \langle \CA_X, \CO_X(-h), \CO_X \rangle,
$$
where $\CA_X = \langle \CO_X(-h), \CO_X \rangle^\perp$. Moreover, $\SS_{\CA_X} \cong \tau[2]$.
\end{corollary}

Note that the behavior of the Serre functor of $\CA_X$ is analogous to that of an Enriques surface.
We are going to show that this is not just a coincidence. But first of all we are going to extend
the construction of subcategory $\CA_X$ to nodal quartic double solids.

\subsection{Nodal case, big resolution}

Now let $D \subset \PP^3$ be a quartic surface with a finite number $N$ of ordinary double points $y_1,\dots,y_N \in \PP^3$.
Then the corresponding double covering $f:X \to \PP^3$ also has $N$ ordinary double points, $x_1 = f^{-1}(y_1)$, \dots, $x_N = f^{-1}(y_N)$.
Therefore, the category $\CA_X$ in this case is singular. If we want to replace it by a smooth category, we have two possibilities.
The first, is to replace the singular threefold $X$ by its blowup in points $x_1,\dots,x_N$,
$$
X' = \Bl_{x_1,\dots,x_N} X.
$$
It is easy to see that $X'$ can also be represented as a double covering.

\begin{lemma}\label{xprime}
Let
$$
Y = \Bl_{y_1,\dots,y_N}\PP^3
$$
be the blowup of $\PP^3$ in points $y_1,\dots,y_N$ and $D' \subset Y$ the proper preimage of $D$.
Then $X'$ is the double covering of $Y$ ramified in $D'$.
\end{lemma}
\begin{proof}
Evident.
\end{proof}

Denote the morphism $X' \to Y$ also by $f$. Denote the exceptional divisors of $Y \to \PP^3$ by $E_1,\dots,E_N$,
and the exceptional divisors of $X' \to X$ by $Q_1,\dots,Q_N$. Each $E_i$ is isomorphic to $\PP^2$ and each $Q_i$
is isomorphic to $\PP^1\times\PP^1$ (since $x_i$ is an ordinary double point). Moreover, $f$ maps $Q_i$ to $E_i$
and is a double covering ramified in a conic.
$$
\xymatrix{
\bigsqcup\limits_{i=1}^N \PP^1\times\PP^1 \ar@{=}[r] &
\bigsqcup\limits_{i=1}^N Q_i \ar[d]_{2:1} \ar@{=}[r] &
X' \ar[rr]^{\Bl_{x_1,\dots,x_N}} \ar[d]_{2:1}^f &&
X \ar[d]_{2:1}^f \\
\bigsqcup\limits_{i=1}^N \PP^2 \ar@{=}[r] &
\bigsqcup\limits_{i=1}^N E_i \ar[r] &
Y \ar[rr]_{\Bl_{y_1,\dots,y_N}} &&
\PP^3
}
$$
The naive replacement of the category $\CA_X$ would be the orthogonal to $\langle \CO_{X'}(-h), \CO_{X'} \rangle$ in $\D^b(X')$.
However, there is a better replacement. Note that the ramification divisor $D'$ lies in the linear system $4h - 2\sum e_i$,
where $e_i$ is the class of the exceptional divisor of $Y$ over the point $y_i \in \PP^3$. Denote one half of it by $H$:
\begin{equation}\label{defH}
H = 2h - \sum e_i.
\end{equation}
Note that
\begin{equation}\label{ky}
K_Y = -2H = 2\sum e_i - 4h.
\end{equation}

\begin{proposition}
Consider the collection of line bundles
$$
(\CE_1,\dots,\CE_{N+2}) = (\CO_X(-h),\{\CO_X(-e_i)\}_{i=1}^{N},\CO_X).
$$
Then the collection $(\CE_1(-H),\dots,\CE_{N+2}(-H),\CE_1,\dots,\CE_{N+2})$ is a full exceptional collection in $D^b(Y)$.
\end{proposition}
\begin{proof}
We start with a standard exceptional collection of a blowup (see~\cite{Or})
$$
\D^b(Y) = \langle \CO_Y(-3h),\CO_Y(-2h),\CO_Y(-h),\CO_Y,\{\CO_{E_i}\}_{i=1}^{N},\{\CO_{E_i}(1)\}_{i=1}^{N} \rangle
$$
and perform a sequence of mutations.

Step 1. Translate the block $\{\CO_{E_i}(1)\}_{i=1}^{N}$ to the left.
Since $(\omega_Y)_{|E_i} \cong \CO_Y(2e_i - 4h)_{|E_i} = \CO_{E_i}(-2)$,
it follows from Lemma~\ref{longmut}
that we will obtain
$$
\D^b(Y) = \langle \{\CO_{E_i}(-1)\}_{i=1}^{N},\CO_Y(-3h),\CO_Y(-2h),\CO_Y(-h),\CO_Y,\{\CO_{E_i}\}_{i=1}^{N} \rangle.
$$

Step 2. Apply the left mutation through the block $\{\CO_{E_i}(-1)\}_{i=1}^{N}$ to $\CO_Y(-3h)$ and $\CO_Y(-2h)$.
Note that
$$
\Ext^t(\CO_{E_i}(-1),\CO_Y(-kh)) =
\Ext^{3-t}(\CO_Y(-kh),\CO_{E_i}(-3))^\vee =
H^{3-t}(E_i,\CO_{E_i}(-3))^\vee =
H^{t-1}(E_i,\CO_{E_i})
$$
(the first equality is the Serre duality on $Y$, in the second we use $\CO_Y(h)_{|E_i} \cong \CO_{E_i}$,
and the third is the Serre duality on $E_i$).
Thus we have one $\Ext^1$ for each $i$. Therefore it is clear that the mutation is given
by the following exact sequence
$$
0 \to \CO_Y(-kh) \to \CO_Y(-kh + \sum_{i=1}^{N} e_i) \to \oplus_{i=1}^{N} \CO_{E_i}(-1) \to 0,
$$
for both $k = 2$ and $k = 3$, that is $\LL_{\langle \{\CO_{E_i}(-1) \}_{i=1}^N \rangle} \CO_Y(-kh) = \CO_Y(-kh + \sum_{i=1}^N e_i)$.
Thus, we obtain
$$
\D^b(Y) = \langle \CO_Y(-3h + \sum_{i=1}^{N} e_i),\CO_Y(-2h + \sum_{i=1}^{N} e_i),\{\CO_{E_i}(-1)\}_{i=1}^{N},\CO_Y(-h),\CO_Y,\{\CO_{E_i}\}_{i=1}^{N} \rangle.
$$
%Note that this collection already has the required form. But for convenience we will make one more mutation.

Step 3. Apply the left mutation through $\CO_Y$ to the block $\{\CO_{E_i}\}_{i=1}^{N}$,
and simultaneously the left mutation through $\CO_Y(-2h + \sum_{i=1}^{N} e_i)$
to the block $\{\CO_{E_i}(-1)\}_{i=1}^{N}$.
Note that
$$
\Ext^\bullet(\CO_Y,\CO_{E_i}) = H^\bullet(E_i,\CO_{E_i}),
$$
so we have one $\Hom$ for each $i$. It follows that the mutation of $E_i$
is given by the following exact sequence
$$
0 \to \CO_Y(-e_i) \to \CO_Y \to \CO_{E_i} \to 0,
$$
that is $\LL_{\CO_Y}\CO_{E_i} = \CO_Y(-e_i)$. Hence
$$
\LL_{\CO_Y(-2h + \sum_{i=1}^N e_i)} \CO_{E_i}(-1) = \LL_{\CO_Y(-H)} \CO_{E_i}(-H) = \CO_Y(-H-e_i)
$$
by Lemma~\ref{tensmut}.
So, finally we obtain
\begin{equation}\label{coly}
\D^b(Y) = \langle \CO_Y(-h-H),\{\CO_Y(-e_i-H)\}_{i=1}^{N},\CO_Y(-H),\CO_Y(-h),\{\CO_Y(-e_i)\}_{i=1}^{N},\CO_Y \rangle
\end{equation}
which proves the Proposition.
\end{proof}

Applying Lemma~\ref{pbec} and Proposition~\ref{serre} we obtain the following

\begin{corollary}\label{caxp}
There is a semiorthogonal decomposition
\begin{equation}\label{sodxp}
\D^b(X') = \langle \CA_{X'}, \CO_{X'}(-h),\{\CO_{X'}(-e_i)\}_{i=1}^{N},\CO_{X'} \rangle,
\end{equation}
where $\CA_{X'} = \langle \CO_{X'}(-h),\{\CO_{X'}(-e_i)\}_{i=1}^{N},\CO_{X'} \rangle^\perp$.
Moreover, $\SS_{\CA_{X'}} \cong \tau[2]$.
\end{corollary}

\subsection{Nodal case, small resolution}

Another alternative is to replace $X$ by its {\em small resolution of singularities}.
Note that an analytic neighborhood of an ordinary double point $x_i \in X$ is analytically isomorphic
to an analytic neighborhood of the vertex of the 3-dimensional quadratic cone. In particular, it has a resolution
of singularities (actually two resolutions!) with the exceptional locus being $\PP^1$. Choosing
such resolution for each $x_i$ and gluing all of them together, we obtain a smooth analytic manifold $X^+$
with a map $\sigma:X^+ \to X$. This manifold does not admit an algebraic structure in general,
however it is always a Moishezon variety, in particular an algebraic space.
Therefore one can speak about its derived category of coherent sheaves.

Note also that since the construction of the minimal resolution $X^+$ involves a choice
of one of two local resolutions for each of the singular points $x_1,\dots,x_N$, there are
$2^N$ different minimal resolutions. Each of them has a semiorthogonal decomposition
analogous to that of $X$.

\begin{proposition}\label{dbxplus}
There is a semiorthogonal decomposition
$$
\D^b(X^+) = \langle \CA_{X^+}, \CO_{X^+}(-h), \CO_{X^+} \rangle,
$$
where $\CA_{X^+} = \langle \CO_{X^+}(-h), \CO_{X^+} \rangle^\perp$.
\end{proposition}
\begin{proof}
Note that $\sigma_*\CO_{X^+} \cong \CO_X$ since ordinary double point is a rational singularity, hence
$$
\Ext^\bullet(\sigma^*\CE,\sigma^*\CE') \cong
\Ext^\bullet(\CE,\sigma_*\sigma^*\CE') \cong
\Ext^\bullet(\CE,\CE'\otimes\sigma_*\CO_{X^+}) \cong
\Ext^\bullet(\CE,\CE').
$$
This shows that $(\CO_{X^+}(-h),\CO_{X^+})$ is an exceptional pair in $\D^b(X^+)$.
Hence it gives the required decomposition.
\end{proof}

\begin{remark}
A natural question to ask is whether the Serre functor of $\CA_{X^+}$ can be written as an involution composed with the shift by 2.
It turns out that the answer is negative. Indeed, since $K_{X^+} = -2h$, we have
$$
\SS_{\CA_{X^+}}^{-1} \cong
\LL_{\langle \CO_{X^+}(-h),\CO_{X^+} \rangle}\circ \SS^{-1}_{\D^b(X^+)} \cong
\LL_{\langle \CO_{X^+}(-h),\CO_{X^+} \rangle}\circ \TT_{\CO_{X^+}(2h)}[-3].
$$
It is easy to see that applied to $\sigma^*\CF$ this functor gives $\sigma^*\tau\CF[-2]$.
On the other hand, if $C$ is an exceptional curve of $X^+$ then
$$
\SS_{\CA_{X^+}}(\CO_{C}(-1)) \cong \CO_{C}(-1)[-3].
$$
The shifts are different, hence the behavior of $\CA_{X^+}$ is more complicated in this case.
\end{remark}

\begin{remark}
If we have a family $X_t$ with smooth generic member and $X_0$ being a nodal double solid,
then one can consider $\CA_{X_0^+}$ as a {\em smooth degeneration} of $\CA_{X_t}$. And the category
$\CA_{X'_0}$ can be thought of as a (categorical) blowup of the category $\CA_{X'_0}$.
\end{remark}

\begin{lemma}\label{dbxp}
There is a fully faithful embedding $\CA_{X^+} \to \CA_{X'}$ which extends
to a semiorthogonal decomposition $\CA_{X'} = \langle \CA_{X^+}, \{ \CO_{Q_i}(-1,0) \}_{i=1}^N \rangle$.
\end{lemma}
\begin{proof}
Note that the blow up of the minimal resolution $X^+$ in the union of the exceptional curves
$C_1 = \sigma^{-1}(x_1)$, \dots, $C_N = \sigma^{-1}(x_N)$ coincides with the blowup $X'$ of $X$.
Denote the exceptional divisors by $Q_1$, \dots, $Q_N$.
Note that $Q_i \cong C_i \times \PP^1$.

The standard semiorthogonal decomposition of the blowup in this case gives
$$
\D^b(X') = \langle \D^b(X^+), \{ \CO_{Q_i}(-1,0), \CO_{Q_i} \}_{i=1}^N \rangle.
$$
Replacing $\D^b(X^+)$ by its decomposition we obtain
$$
\D^b(X') = \langle \CA_{X^+}, \CO_{X'}(-h), \CO_{X'}, \{ \CO_{Q_i}(-1,0), \CO_{Q_i} \}_{i=1}^N \rangle.
$$
Now we note that the sheaves $\CO_{Q_i}(-1,0)$ are completely orthogonal to $\CO_{X'}(-h)$ and $\CO_{X'}$.
Indeed, since the latter are pullbacks via $\sigma':X' \to X$, it suffices to observe that
$\sigma'_*\CO_{Q_i}(-1,0) = 0$, since $\sigma'$ contracts each of $Q_i$ to a point and the sheaf
$\CO_{Q_i}(-1,0)$ on $Q_i$ is acyclic. Therefore we can move all $\CO_{Q_i}(-1,0)$ to the left.
We obtain a semiorthogonal decomposition
$$
\D^b(X') = \langle \CA_{X^+}, \{ \CO_{Q_i}(-1,0) \}_{i=1}^N, \CO_{X'}(-h), \CO_{X'}, \{ \CO_{Q_i} \}_{i=1}^N \rangle.
$$
Finally, mutating $\CO_{Q_i}$ through $\CO_{X'}$ we obtain
$$
\D^b(X') = \langle \CA_{X^+}, \{ \CO_{Q_i}(-1,0) \}_{i=1}^N, \CO_{X'}(-h), \{ \CO_{X'}(-e_i) \}_{i=1}^N, \CO_{X'} \rangle.
$$
Comparing it with~\eqref{sodxp} we deduce the claim.
\end{proof}

\section{Quartic symmetroids and special double solids}\label{s-symm}

In this section we discuss the quartic double solids associated with quartic symmetroids and state the main result of the paper.
The following construction is very classical. For more details see e.g.~\cite{Co}.

Assume $V$ and $W$ are vector spaces, both of dimension $4$, and choose a generic embedding
$$
s:W \subset S^2V^\vee.
$$
Alternatively, $s$ can be considered as a tensor in $W^\vee \otimes S^2V^\vee$.
The latter space can be considered as the space of global sections of various natural vector bundles.
Considering zero loci of these sections we obtain some natural varieties associated with $s$.
For example, considering the line bundle $\CO(1,2)$ on $\PP(W)\times\PP(V)$ we obtain
a family of quadrics
$$
Q \subset \PP(W)\times\PP(V).
$$
The degeneration locus of this family $D \subset \PP(W)$ is the zero locus of $\det s \in S^4W^\vee \otimes (\det V^\vee)^2$,
hence is a quartic surface. Quartic surfaces obtained in this way are called {\sf quartic symmetroids}.
If $s$ is generic then $D$ has 10 ordinary double points. The blowup $D'$ of $D$ can be realized
as the zero locus of $s$ considered as a global section of the vector bundle $W^\vee\otimes\CO(1,1)$
on $\PP(V)\times\PP(V)$ (the map $D' \to D$ takes a point $(v,v') \in D'$ to the unique $w \in \PP(W)$
such that quadric $Q_w$ is singular at $v$). On the other hand, the transposition of factors
$\iota:\PP(V)\to\PP(V)$ acts on K3-surface $D'$ without fixed points, hence
$$
S = D'/\iota
$$
is an Enriques surface.

%
%
%\item the bundle $W^\vee\otimes\CO(1,1)$ on $\PP(V)\times\PP(V)$ gives a K3-surface $D' \subset \PP(V)\times\PP(V)$;
%\item the bundle $S^%
%
%%There are several geometric objects associated with $s$:
%\begin{itemize}%
%
%
%\item the family of quadrics $Q \subset \PP(W)\times \PP(V)$; it is a divisor of bidegree $(1,2)$;
%\item the locus of pairs of points in $\PP(V)$ orthogonal with respect to all quadrics $D' \subset \PP(V)\times \PP(V)$;
%it is a complete intersection of $4$ divisors of bidegree $(1,1)$;
%\item the locus of lines contained in a pencil of quadrics $S \subset \Gr(2,V)$;
%\item the relative scheme of lines of the family of quadrics $M \subset \Gr(2,V)\times\PP(W)$.
%\item the locus of degenerate quadrics $D \subset \PP(W)$; it is a divisor of degree $4$;
%%\item the locus of singular points of quadrics $D' \subset \PP(W)\times \PP(V)$;
%\end{itemize}
%If the embedding $s$ is chosen sufficiently generic then $Q$ is smooth
%and $D$ is a nodal quartic surface. The number of nodes equals 10,
%and the quartic $D$ is known as the {\sf quartic symmetroid}\/ associated with $s$.
%Moreover, it is known that $D'$ is the resolution of singularities of $D$,
%and that $S$ is the Enriques surface isomorphic to the quotient of $D'$
%by the involution induced by the permutation of factors of $\PP(V)\times\PP(V)$.

\begin{remark}
Enriques surfaces associated in this way with quartic symmetroids are known
as {\sf nodal Enriques surfaces} (the name is slightly misleading, since
they are smooth), or {\sf Reye congruences}. The family of nodal Enriques
surfaces is known to be a divisor in the family of all Enriques surfaces.
\end{remark}

%It is clear that $Z_0$ is the quartic symmetroid with 10 nodes.
%Moreover, it is easy to see that $Z$ is its resolution of singularities.
%It is also known that $S$ is the associated Enriques surface (the quotient of $Z$ by a fixed points free involution).

Our goal is the following

\begin{conjecture}
Let $X'$ be the blowup of the nodal quartic double solid $X$ associated with the quartic symmetroid $D \subset \PP^3$
and $S = D'/\iota$ --- the corresponding Enriques surface.
Then
$$
\CA_{X'} \cong \D^b(S),
$$
where subcategory $\CA_{X'} \subset \D^b(X')$ is defined in Corollary~{\rm\ref{caxp}}.
\end{conjecture}

Instead of proving this conjecture we concentrate on proving the following
slightly simpler result. Note that both categories $\CA_{X'}$ and $\D^b(S)$
have a completely orthogonal exceptional collections of length~10:
for $\CA_{X'}$ it is given by the sheaves $\CO_{Q_i}(-1,0)$ (see Lemma~\ref{dbxp}),
and for $\D^b(S)$ it is given by the line bundles $\CO_S(-F_i^+)$ (see Theorem~\ref{zube}).
The orthogonal complements of these collections are the categories
$\CA_{X^+}$ and $\CA_S$ respectively.
The main result of the paper is the following

%
%Recall that the derived category of any Enriques surface $S$ has a completely orthogonal
%exceptional collection of 10 line bundles corresponding to 10 elliptic pencils on $S$ ---
%if $F_i$, $F'_i$ are the multiple fibers of the $i$-th pencil then
%$$
%\{ \CO_S(F_i) \}_{i=1}^{10}
%$$
%is such a collection. It extends naturally to a semiorthogonal decomposition

\begin{theorem}\label{cxpcs}
Let $X^+$ be a small resolution of singularities of the nodal quartic double solid $X$ associated with the quartic symmetroid $D \subset \PP^3$
and $S$ the corresponding Enriques surface.
Then
$$
\CA_{X^+} \cong \CA_S,
$$
where subcategories $\CA_{X^+} \subset \D^b(X^+)$ and $\CA_S \subset \D^b(S)$ are defined in
Proposition~{\rm\ref{dbxplus}} and~\eqref{cas} respectively.
\end{theorem}

%In view of decompositions~\eqref{dbxp} and~\eqref{dbs}, the categories $\CA_{X^+}$ and $\CA_S$ are the orthogonal complements
%of an exceptional collection of 10 exceptional objects in $\CA_{X'}$ and $\D^b(S)$ respectively. So, this Theorem is a good
%evidence for the Conjecture.

\begin{remark}
Note that there are $2^{10}$ ways to choose one multiple fiber
in each of 10 elliptic pencils on~$S$, precisely the same number
as the number of small resolutions of singularities of $X$.
We will show below that there is a bijection between these sets
and the equivalence of Theorem~\ref{cxpcs} is for proved compatible choices.
\end{remark}

\begin{remark}
Note that different small resolutions of $X$ are related by flops.
These flops induce equivalences of the corresponding subcategories $\CA_{X^+}$.
On the other hand, different choices of multiple fibers in elliptic pencils on $S$
lead to subcategories $\CA_S$ related by appropriate mutation functors.
\end{remark}

The proof of the Theorem is given in the next section.

\section{The relative scheme of lines}\label{s-proof}

The main idea of the proof is to analyze carefully the derived category of
the relative scheme of lines $M$ on the family $Q \to \PP^3$ of quadrics associated with $s$.
One one hand, we show that $M$ is the blowup of $\Gr(2,V)$ in the Enriques surface $S$,
whereof one realizes $\D^b(S)$ as a component of $\D^b(M)$. On the other hand,
there is a general description of the relative scheme of lines~\cite{K10},
one component of which is the derived category of $X^+$. Finally, a sequence
of mutations allows to transform one of these decompositions of $\D^b(M)$
to the other and establishes the required equivalence.

\subsection{Relation to the Enriques surface}

First of all we will need another description of the Enriques surface associated with the quartic symmetroid.
For this consider the product of the Grassmannians $\Gr(2,V) \times \Gr(2,W)$. Denote the tautological subbundles
on those as $\CU$ and $\CU_W$ respectively. Then $s$ can be thought of as a global section of the vector
bundle $S^2\CU^\vee \boxtimes \CU_W$. Consider its zero locus $\TS \subset \Gr(2,V)\times\Gr(W)$.
Set-theoretically, $\TS$ is the set of all pairs $(L,K)$ of twodimensional subspaces $L \subset V$,
$K \subset W$ such that the line $\PP(L)$ is contained in all the quadrics in the pencil $\PP(K) \subset \PP(W)$.
It is known (see e.g.~\cite{Be}, Chapter VIII) that $\TS$ is isomorphic to the Enriques
surface associated with~$s$ (the map $D' \to \TS$ takes a point $(v,v') \in D' \subset \PP(V)\times \PP(V)$
to the point $(L,K) \in \Gr(2,V)\times\Gr(2,W)$, where $L$ is the linear span of $v$ and $v'$
and $K$ consists of all $w \in \PP(W)$ such that $Q_w$ contains both $v$ and~$v'$) and that
the projection onto $\Gr(2,V)$ gives an embedding of $S = \TS$ into $\Gr(2,V)$.

Below we will need a resolution of the structure sheaf of $S$ on $\Gr(2,V)$.

\begin{lemma}
The structure sheaf of $S$ in $\Gr(2,V)$ has the following resolution
\begin{equation}\label{os}
0 \to S^2\CU(-3g) \to W^\vee\otimes\CO_{\Gr(2,V)}(-3g) \to \CO_{\Gr(2,V)} \to \CO_S \to 0.
\end{equation}
\end{lemma}
\begin{proof}
By definition of $\TS$ we have the following Koszul resolution for $\CO_{\tilde{S}}$:
%\begin{multline*}
$$
\begin{array}{rcl}
0 \to \CO(-6,-3) \to
S^2\CU(-4)\boxtimes\CU_W(-2) \to \\
{S^2\CU(-3)\boxtimes S^2\CU_W(-1)} & \oplus &  (S^4\CU(-2) \oplus \CO(-4))\boxtimes\CO(-2) \to \\
{\CO(-3)\boxtimes S^3\CU_W} & \oplus & (S^4\CU(-1) \oplus S^2\CU(-2))\boxtimes \CU_W(-1) \to \\
S^2\CU(-1)\boxtimes S^2\CU_W & \oplus & (S^4\CU \oplus \CO(-2)) \boxtimes \CO(-1) \to \\
&& \qquad\qquad\qquad\qquad\qquad S^2\CU\boxtimes\CU_W \to
{\CO} \to \CO_{\tilde{S}} \to 0.
\end{array}
$$
%\end{multline*}
Since the projection to $\Gr(2,V)$ identifies $\tilde{S}$ with $S$,
we can take the pushforward of the above resolution to get a resolution of $\CO_S$.
Note that the bundles $\CU_W$, $\CO(-1)$, $S^2\CU_W$, $\CU_W(-1)$, $\CO(-2)$, $\CU_W(-2)$ and $\CO(-3)$ on $\Gr(2,W)$
are acyclic, the nontrivial pushforward will only come from the terms
$S^2\CU(-3)\boxtimes S^2\CU_W(-1)$, $\CO(-3)\boxtimes S^3\CU_W$, and $\CO$.
Since
$$
H^2(\Gr(2,W),S^2\CU_W(-1)) = \C,\qquad
H^2(\Gr(2,W),S^3\CU_W) = W^\vee,\qquad
H^0(\Gr(2,W),\CO) = \C,
$$
and the other cohomology vanishes, we obtain the desired resolution.
\end{proof}

Now consider the relative scheme of lines $M$ for the family of quadrics $Q \to \PP(W)$.
By definition $M$ is a subscheme in $\Gr(2,V)\times \PP(W)$, consisting of all pairs $(L,w)$,
where $L$ is a 2-dimensional subspace of $V$ and $w \in \PP(W)$, such that the line $\PP(L)$
lies on the quadric $Q_w \subset \PP(V)$. Thus, it is the zero locus of a section
of the vector bundle $S^2\CU^\vee\otimes\CO_{\PP(W)}(1)$ on $\Gr(2,V)\times\PP(W)$,
given by $s$. Since $s$ is generic the section is regular.
Denoting the positive generator of $\Pic(\Gr(2,V))$ by $\CO(g)$
we obtain the following

\begin{lemma}
The structure sheaf of $M$ in $\Gr(2,V)\times\PP(W)$ has the following resolution
\begin{equation}\label{om}
0 \to \CO(-3g-3h) \to S^2\CU(-g-2h) \to S^2\CU(-h) \to \CO \to \CO_M \to 0.
\end{equation}
\end{lemma}
\begin{proof}
This is just the Koszul resolution.
\end{proof}

Now we are ready to show that map $\pi:M \to \Gr(2,V)$ is the blowup of $S \subset \Gr(2,V)$.

\begin{lemma}
The map $\pi:M \to \Gr(2,V)$ is the blowup of $S \subset \Gr(2,V)$.
Moreover, if $e$ is the class of the exceptional divisor then $h = 3g - e$.
\end{lemma}
\begin{proof}
First, let us describe the fibers of $\pi$. It is clear that the fiber over a point of $\Gr(2,V)$
corresponding to a subspace $L \subset V$ is the linear subspace of $\PP(W)$ consisting of quadrics,
containing the line $\PP(L)$. This is a condition of codimension 3, hence the generic fiber is a point.
Moreover, the fiber is nontrivial if there is a pencil of quadrics containing $\PP(L)$, that is if $L$
is in the image of the map $S = \TS \to \Gr(2,V)$. This shows that $\pi$ is the blowup of an ideal
supported at $S$. It remains to check that this is the ideal of~$S$.
For this consider the sheaf $\pi_*(\CO_M(h))$. Since $h$ is very ample over $\Gr(2,V)$
this sheaf is a twist of the ideal in question.
To compute $\pi_*\CO_M(h))$ we use~\eqref{om}. It gives
$$
0 \to S^2\CU \to W^\vee\otimes\CO \to \pi_*\CO_M(h) \to 0.
$$
Comparing with~\eqref{os} we see that $\pi_*\CO_M(h) \cong I_S(3g)$.
Thus $M$ is the blowup of $S$. Moreover, it follows also that $h = 3g - e$.
\end{proof}

Denote the exceptional divisor of $\pi$ by $E$, and the restriction of $\pi$ to $E$ by $p$.
Let also $i:E \to M$ be the embedding. Using~\cite{Or} we obtain the following

\begin{corollary}\label{dbmpi}
We have the following semiorthogonal decomposition
\begin{equation}\label{som0}
\D^b(M) = \langle \pi^*(\D^b(\Gr(2,V))), i_*p^*(\D^b(S)) \rangle.
\end{equation}
\end{corollary}

Since $\D^b(\Gr(2,V))$ is generated by an exceptional collection,
we see that $\D^b(M)$ consists of several exceptional objects and the category $\CA_S$.
We will continue using this line in subsection~\ref{ss-mut}, but now
we will look at $M$ from the other point of view.

\subsection{Relation to the double solid}

So far we considered $M$ as a blowup of $\Gr(2,V)$.
Now consider the projection $\rho:M \to \PP(W)$.
By definition of $M$ it is the relative scheme of lines
for the family of quadrics $Q \to \PP(W)$. Therefore the generic fiber of $M$
is a disjoint union of two conics; while over the divisor $D$ of degenerate quadrics
the fiber is a single conic, and over 10 points corresponding to quadrics of corank~2
the fiber is a union of two planes intersecting in a point. Note that the Stein
factorization for the map $\rho:M \to \PP(W)$ is
$$
\xymatrix{M \ar[rr]^\mu \ar[dr]_\rho && X \ar[dl]^f \\ & \PP(W)},
$$
where $f:X \to \PP(W)$ is the double covering ramified in $D$, that is $X$ is the special
double solid we are interested in.

The derived category of a relative scheme of lines for a family of 2-dimensional quadrics
was described in~\cite{K10}. To state the answer obtained therein we need some notation to be introduced.
First of all, for each point $y_i$ corresponding to a quadric of corank 2 choose one of two planes
in the fiber of $M$ over $y_i$. Denote it by $\Sigma_i^+$, and the other by $\Sigma_i^-$.
Further, consider the sheaf $\CB_0$ of even parts of Clifford algebras on $\PP(W)$
associated with the family of quadrics $Q \to \PP(W)$ (see~\cite{K08a} for details).
Let $\D^b(\PP(W),\CB_0)$ denote the derived category of sheaves of $\CB_0$-modules on $\PP(W)$.

\begin{theorem}\label{dbm}
For each choice of $10$ planes $\Sigma^+_i$ there exists a minimal resolution $X^+$ of $X$
and a semiorthogonal decomposition
%There is a semiorthogonal decomposition
$$
\D^b(M) = \langle \D^b(X^+), \D^b(\PP(W),\CB_0), \{ \CO_{\Sigma_i^+} \}_{i=1}^{10} \rangle.
$$
%where $X^+$ is a small resolution of singularities of $X$.
\end{theorem}

We will also need a description of the embedding functors for the first two components of this decomposition.
First of all, the functor $\D^b(X^+) \to \D^b(M)$ can be constructed as follows.
One can check (see~\cite{K10}) that the normal bundle to any of the planes $\Sigma_i^+$ is $\CO(-1) \oplus \CO(-1)$.
Hence one can perform a flip $M \dashrightarrow M^+$ in all these planes. The obtained Moishezon variety $M^+$
has a structure of a $\PP^1$-fibration $\mu_+:M^+ \to X^+$  over $X^+$. In particular, the pullback functor $\mu_+^*$
gives an embedding $\D^b(X^+) \to \D^b(M^+)$. Further, the flip $M \dashrightarrow M^+$ can be represented as
a composition of a blowup $\xi:\TM \to M$ and a blowdown $\xi_+:\TM \to M^+$. The composition
of functors $\xi_*\xi_+^*:\D^b(M^+) \to \D^b(M)$ is also fully faithful by a result of Bondal and Orlov (see~\cite{BO95}).
Composing we will obtain an embedding $\Phi_0 = \xi_*\xi_+^*\mu_+^*:\D^b(X^+) \to \D^b(M)$. What we will need further is the following observation

\begin{lemma}\label{phi0}
Let $f_+:X_+ \to \PP(W)$ be the composition of the contraction $X^+ \to X$ and the double covering $X \to \PP(W)$.
%Let $\rho:M \to \PP(W)$ be the projection.
For any object $\CF \in \D^b(\PP(W))$ one has
$$
\Phi_0(f_+^*\CF) \cong \rho_+^*\CF.
$$
Moreover, the functor $\Phi_0$ is $\PP(W)$-linear, that is
$$
\Phi_0(\CG\otimes f_+^*\CF) \cong \Phi_0(\CG)\otimes\rho_+^*\CF.
$$
\end{lemma}
\begin{proof}
Follows immediately from the projection formula.
\end{proof}

Now consider the second component.
Note that besides the sheaf of even parts of Clifford algebras $\CB_0$ one has also
the sheaf $\CB_1$ of odd parts of Clifford algebras on $\PP(W)$.
As sheaves on $\PP(W)$ they have the following structure
$$
\CB_0 = \CO \oplus \Lambda^2V(-h) \oplus \Lambda^4V(-2h),\qquad
\CB_1 = V\otimes\CO \oplus \Lambda^3V(-h).
$$
This pair of sheaves naturally extends to a sequence of sheaves $\CB_k$ of $\CB_0$-modules
defined by
$$
\CB_{2k} = \CB_0(kh),
\qquad
\CB_{2k+1} = \CB_1(kh).
$$
For each $k \in \ZZ$ consider the morphism $\CU\boxtimes\CB_{k-1} \to \CB_k$ of sheaves of $\CB_0$-modules
on $\Gr(2,V)\times\PP(W)$ induced by the embedding $\CU\subset V\otimes\CO$ and the action $V\otimes\CB_k \to \CB_{k+1}$.
Let $\alpha:M \to \Gr(2,V) \times \PP(W)$ be the embedding.
Another result of~\cite{K10} is the following

\begin{theorem}[\cite{K10}]\label{phi}
There are isomorphisms
$\Coker(\CU\boxtimes\CB_{k-1} \to \CB_k) \cong \alpha_*\CS_k$,
where the sheaves $\CS_k$ are defined as
\begin{equation}\label{c2k}
0 \to \CO(kh) \to \CS_{2k} \to \CO(g + (k-1)h) \to 0,\qquad
\CS_{2k+1} = (V/\CU)(kh),
\end{equation}
and the extension is nontrivial. Moreover, we have
\begin{equation}\label{rhosk}
\rho_*(\CS_k(-g)) = 0
\end{equation}
for all $k \in \ZZ$.
Finally, the functor $\Phi_1:\D^b(\PP(W),\CB_0) \to \D^b(M)$, $\CF \mapsto \CS_0\otimes_{\CB_0} \rho^*\CF$ is fully faithful
and
$$
\Phi_1(\CB_k) \cong \CS_k.
$$
\end{theorem}

On the other hand, it follows from~\cite{K08a}, Theorem~5.5 that $\D^b(\PP(W),\CB_0)$ is generated by an exceptional collection
$\CB_{k-3},\CB_{k-2},\CB_{k-1},\CB_k$, where $k$ can be taken arbitrarily. Indeed there might be an additional component,
the derived category of the intersection of all quadrics in $\PP(V)$ parameterized by $W$. But since $s$ is generic this
is the intersection of 4 generic quadrics in $\PP^3$, and so is empty. Combining this result with the Theorem above
we obtain the following

\begin{corollary}\label{phi1}
The subcategory $\Phi_1(\D^b(\PP(W),\CB_0)) \subset \D^b(M)$ is generated by an exceptional collection
$$
\langle \CS_{k-3},\CS_{k-2},\CS_{k-1},\CS_k \rangle,
$$
where $k$ can be taken arbitrarily.
\end{corollary}

Combining the above results we can write down the following semiorthogonal decomposition
\begin{equation}\label{dbmpi1}
\D^b(M) = \langle \Phi_0(\D^b(X^+)), \CS_{-1}, \CS_0, \CS_1, \CS_2,  \{ \CO_{\Sigma_i^+} \}_{i=1}^{10} \rangle.
\end{equation}
%Indeed, the first three components replace $\D^b(X^+)$ (see Proposition~\ref{dbxplus}),
%and the next four replace $\D^b(\PP(W),\CB_0)$ (see Corollary~\ref{phi1})
%in the decomposition of Theorem~\ref{dbm}.

\subsection{The mutations}\label{ss-mut}

Now our strategy is the following. We start with decomposition~\eqref{som0} of $\D^b(M)$
and then apply a sequence of mutations to get decomposition~\eqref{dbmpi1}.
As a result we will obtain a relation between $\CA_{X^+}$ and $\D^b(S)$.
More precisely, instead of the first component of~\eqref{som0}
we substitute the following exceptional collection on $\Gr(2,V)$:
$$
\D^b(\Gr(2,V)) = \langle \CO(-2g), \CO(-g), V/\CU(-g), \CO, V/\CU, \CO(g) \rangle.
$$
Plugging it into~\eqref{som0} and denoting
\begin{equation}\label{psi0}
\Psi_0 = i_*p^*: \D^b(S) \to \D^b(M),
\end{equation}
we obtain a semiorthogonal decomposition
\begin{equation}\label{som1}
\D^b(M) = \langle \CO(-2g), \CO(-g), V/\CU(-g), \CO, V/\CU, \CO(g), \Psi_0(\D^b(S)) \rangle.
\end{equation}

Now we perform a sequence of mutations.

Step 1. Apply the left mutation through $V/\CU$ and $\CO(g)$ to $\Psi_0(\D^b(S))$. We obtain
\begin{equation}\label{som2}
\D^b(M) = \langle \CO(-2g), \CO(-g), V/\CU(-g), \CO, \Psi_1(\D^b(S)), V/\CU, \CO(g) \rangle,
\end{equation}
where
\begin{equation}\label{psi1}
\Psi_1 = \LL_{\langle V/\CU, \CO(g)\rangle}\circ i_*p^*: \D^b(S) \to \D^b(M).
\end{equation}

Step 2. Translate bundles $V/\CU$ and $\CO(g)$ to the left. By Lemma~\ref{longmut} we obtain
\begin{equation}\label{som3}
\D^b(M) = \langle V/\CU(-g-h), \CO(-h), \CO(-2g), \CO(-g), V/\CU(-g), \CO, \Psi_1(\D^b(S)) \rangle
\end{equation}
since $K_M = - 4g + e = - g - h$.

Step 3. Apply the left mutation through $V/\CU(-g-h)$ and $\CO(-h)$ to $\CO(-2g)$.

\begin{lemma}
We have $\Ext^\bullet(\CO(-h),\CO(-2g)) = \Ext^\bullet(V/\CU(-g-h),\CO(-2g)) = 0$.
\end{lemma}
\begin{proof}
We have to compute
$$
H^\bullet(M,\CO(h-2g)) = H^\bullet(M,\CO(g-e)),
\qquad\text{and}\qquad
H^\bullet(M,\CU^\perp(h-g)) = H^\bullet(M,\CU^\perp(2g-e)),
$$
where $\CU^\perp := (V/\CU)^\vee$.
Since $\pi_*(\CO(-e)) = I_S$, the projection formula implies that this is the same as
$H^\bullet(\Gr(2,V),I_S(g))$ and $H^\bullet(\Gr(2,M),I_S\otimes\CU^\perp(2g))$ respectively.
Using~\eqref{os} as a resolution of $I_S$ and Borel--Bott--Weil on $\Gr(2,V)$
we obtain the required vanishing.
\end{proof}

Because of the orthogonality proved above and by Lemma~\ref{perpmut} we obtain
\begin{equation}\label{som4}
\D^b(M) = \langle \CO(-2g), V/\CU(-g-h), \CO(-h), \CO(-g), V/\CU(-g), \CO, \Psi_1(\D^b(S)) \rangle.
\end{equation}

Step 4. Apply the left mutation through $\CO(-h)$ to $\CO(-g)$.

\begin{lemma}
We have $\Ext^\bullet(\CO(-h),\CO(-g)) = \C[-1]$.
\end{lemma}
\begin{proof}
We have to compute $H^\bullet(M,\CO(h-g)) = H^\bullet(M,\CO(2g-e))$.
Since $\pi_*(\CO(-e)) = I_S$, the projection formula implies that this is the same as
$H^\bullet(\Gr(2,V),I_S(2g))$. By Borel--Bott--Weil we have $H^\bullet(\Gr(2,V),\CO(-g)) = 0$
and $H^\bullet(\Gr(2,V),S^2\CU(-g)) = \C[-2]$. Using~\eqref{os} as a resolution of $I_S$ we
deduce the claim.
\end{proof}

It follows that $\LL_{\CO(-h)}(\CO(-g))$ is the unique extension of $\CO(-h)$ by $\CO(-g)$.
By Theorem~\ref{phi} it is isomorphic to $\CS_0(-g)$. Thus we obtain
\begin{equation}\label{som5}
\D^b(M) = \langle \CO(-2g), V/\CU(-g-h), \CS_0(-g), \CO(-h), V/\CU(-g), \CO, \Psi_1(\D^b(S)) \rangle.
\end{equation}

Step 5. Apply the right mutation through $V/\CU(-g)$ to $\CO(-h)$.

\begin{lemma}
We have $\Ext^\bullet(\CO(-h),V/\CU(-g)) = 0$.
\end{lemma}
\begin{proof}
We have to compute $H^\bullet(M,V/\CU(h-g)) = H^\bullet(M,V/\CU(2g-e))$.
Since $\pi_*(\CO(-e)) = I_S$, the projection formula implies that this is the same as
$H^\bullet(\Gr(2,M),I_S\otimes V/\CU(2g))$.
Using~\eqref{os} as a resolution of $I_S$ and Borel--Bott--Weil on $\Gr(2,V)$
we obtain the required vanishing.
\end{proof}

Because of the orthogonality proved above and by Lemma~\ref{perpmut} we obtain
\begin{equation}\label{som6}
\D^b(M) = \langle \CO(-2g), V/\CU(-g-h), \CS_0(-g), V/\CU(-g), \CO(-h), \CO, \Psi_1(\D^b(S)) \rangle.
\end{equation}

Step 6. Translate bundle $\CO(-2g)$ to the right. By Lemma~\ref{longmut} we obtain
\begin{equation}\label{som7}
\D^b(M) = \langle V/\CU(-g-h), \CS_0(-g), V/\CU(-g), \CO(-h), \CO, \Psi_1(\D^b(S)), \CO(h-g) \rangle.
\end{equation}

Step 7. Apply the right mutation through $\CO(h-g)$ to $\Psi_1(D^b(S))$. We obtain
\begin{equation}\label{som8}
\D^b(M) = \langle V/\CU(-g-h), \CS_0(-g), V/\CU(-g), \CO(-h), \CO, \CO(h-g), \Psi_2(\D^b(S)) \rangle,
\end{equation}
where
\begin{equation}\label{psi2}
\Psi_2 = \RR_{\CO(h-g)}\circ \LL_{\langle V/\CU, \CO(g)\rangle}\circ i_*p^*: \D^b(S) \to \D^b(M).
\end{equation}

Step 8. Apply the left mutation through $\CO$ to $\CO(h-g)$. Analogously to step 4 we have $\LL_\CO(\CO(h-g)) = \CS_0(h-g) = \CS_2(-g)$.
Thus we obtain
\begin{equation}\label{som9}
\D^b(M) = \langle V/\CU(-g-h), \CS_0(-g), V/\CU(-g), \CO(-h), \CS_2(-g), \CO, \Psi_2(\D^b(S)) \rangle.
\end{equation}

Step 9. Finally, apply the left mutation through $\CO(-h)$ to $\CS_2(-g)$.

\begin{lemma}
We have $\Ext^\bullet(\CO(-h),\CS_2(-g)) = 0$.
\end{lemma}
\begin{proof}
Follows immediately from~\eqref{rhosk}.
\end{proof}

Because of the orthogonality and by Lemma~\ref{perpmut} we obtain
$$
\D^b(M) = \langle V/\CU(-g-h), \CS_0(-g), V/\CU(-g), \CS_2(-g), \CO(-h), \CO, \Psi_2(\D^b(S)) \rangle.
$$
Finally, rewriting $V/\CU = \CS_1$, $V/\CU(-h) = \CS_{-1}$ (see Theorem~\ref{phi}) we see that we have obtained
\begin{equation}\label{som10}
\D^b(M) = \langle \CS_{-1}(-g), \CS_0(-g), \CS_1(-g), \CS_2(-g), \CO(-h), \CO, \Psi_2(\D^b(S)) \rangle.
\end{equation}

\subsection{Multiple elliptic fibers}

Now we have to take into account the structure of $\D^b(S)$, namely the semiorthogonal decomposition
\begin{equation}\label{dbs}
\D^b(S) = \langle \{ \CO_S(-F_i^+) \}_{i=1}^{10}, \CA_S \rangle,
\end{equation}
where $F_i^\pm$ are the multiple fibers of the 10 elliptic pencils on $S$ and $\CA_S = {}^\perp\langle \{ \CO_S(-F_i^+) \}_{i=1}^{10} \rangle$.
It turns out that the functor $\Psi_2$ relates the multiple elliptic fibers $F_i^\pm$
to the planes $\Sigma^\pm_i$ of $M$.

Let $y_i \in \PP(w)$ be one of 10 nodes of the symmetroid $D$.
Then as we know the quadric $Q_{y_i}$ is a union of two planes,
and $\rho^{-1}(y_i) = \Sigma_i^+ \cup \Sigma_i^-$ is also a union of two planes.

\begin{lemma}
The exceptional divisor $E$ of the blowup $\pi:M \to \Gr(2,V)$ intersects $\Sigma_i^\pm$ along an elliptic curve.
Its image $F_i^\pm = \pi(\Sigma_i^\pm \cap E) \subset S$ is a multiple elliptic fiber of $S$.
\end{lemma}
\begin{proof}
Note that $e = 3g - h$, hence $\CO(e)_{|\Sigma_i^\pm} = \CO(3)$.
On the other hand, $\Sigma_i^\pm$ is not contained in $E$, since otherwise
we would have $\PP^2 = \pi(\Sigma_i^\pm) \subset S$ which is impossible.
Hence the intersection is a plane cubic. But the only plane cubic curves
on $S$ are the multiple elliptic fibers.
\end{proof}

\begin{proposition}\label{FS}
Let $F_i^\pm = \pi(E \cap \Sigma^\pm_i)$ be a multiple elliptic fiber of $S$. Then
$$
\Psi_2(\CO_S(2g - F_i^\pm)) \cong \CO_{\Sigma_i^\pm}(-1).
$$
\end{proposition}
\begin{proof}
To unburden the notation we will write $\Sigma$ instead of $\Sigma_i^\pm$ and $F$ instead of $F_i^\pm$.
Now consider several exact sequences. First, note that $\pi(\Sigma) = \Gr(2,3) \subset \Gr(2,V)$ is the zero locus
of a regular section of a vector bundle $\CU^\vee$; hence we have the following Koszul resolution on $\Gr(2,V)$
$$
0 \to \CO(-g) \to \CU \to \CO \to \CO_{\pi(\Sigma)} \to 0.
$$
Identifying $\CU = \CU^\vee(-g)$ and pulling back to $M$ we obtain
\begin{equation}\label{ppsk}
0 \to \CO(-g) \to \CU^\vee(-g) \to \CO \to \CO_{\pi^{-1}(\pi(\Sigma))} \to 0.
\end{equation}
On the other hand, the preimage $\pi^{-1}(\pi(\Sigma))$ clearly has two components,
one is $\Sigma$ itself, and the other is $p^{-1}(\pi(\Sigma) \cap S) = p^{-1}(F)$.
The components intersect along the cubic curve $\Sigma \cap p^{-1}(F)$.
So, the structure sheaf of $\pi^{-1}(\pi(\Sigma))$ fits into the following exact sequence
\begin{equation}\label{ppsu}
0 \to \CO_{\Sigma}(-3) \to \CO_{\pi^{-1}(\pi(\Sigma))} \to i_*p^*\CO_{F} \to 0.
\end{equation}
Also we have the following Koszul resolutions
\begin{equation}\label{ipf}
0 \to i_*p^*\CO_S(-F) \to i_*p^*\CO_S \to i_*p^*\CO_{F} \to 0
\end{equation}
and
\begin{equation}\label{ips}
0 \to \CO(-e) \to \CO \to i_*p^*\CO_S \to 0.
\end{equation}
Combining~\eqref{ppsk}, \eqref{ppsu}, \eqref{ipf}, and~\eqref{ips} and twisting everything by $\CO(2g)$ we obtain a bicomplex
$$
\xymatrix{
&&& \CO(2g-e) \ar[r] \ar[d] & \CO_{\Sigma}(-1) \ar[d] \ar[r] & 0\\
0 \ar[r] & \CO(g) \ar[r] \ar[d] & \CU^\vee(g) \ar[r] \ar[d] & \CO(2g) \ar[r] \ar[d] & \CO_{\pi^{-1}(\pi(\Sigma))}(2g) \ar[r] \ar[d] & 0 \\
& 0 \ar[r] & i_*p^*\CO_S(2g-F) \ar[r] & i_*p^*\CO_S(2g) \ar[r] & i_*p^*\CO_{F}(2g) \ar[r] & 0
}
$$
Here the middle row is~\eqref{ppsk}, the bottom row is~\eqref{ipf},
the third column is~\eqref{ips}, and the fourth column is~\eqref{ppsu},
everything twisted by $\CO(g)$ as it was already mentioned.

Since the bottom two rows of the bicomplex are exact, the total complex is quasiisomorphic to the top row,
which can be rewritten as
\begin{equation}\label{rpi}
0 \to \CO(h-g) \to \CO_{\Sigma}(-1) \to 0
\end{equation}
(the map here is just the restriction
$\CO(h-g) \to \CO(h-g)_{|\Sigma} \cong \CO(2g-e)_{|\Sigma} \cong \CO_{\Sigma}(-1)$).
On the other hand, since the right two columns are exact, the total complex is quasiisomorphic to
\begin{equation}\label{lf}
0 \to \CO(g) \to \CU^\vee(g) \to i_*p^*\CO_S(2g-F) \to 0.
\end{equation}
Combining these observations we conclude that complexes~\eqref{rpi} and~\eqref{lf} are quasiisomorphic.

Now we claim that complex~\eqref{rpi} is quasiisomorphic to $\LL_{\CO(h-g)}(\CO_{\Sigma}(-1))$
and complex~\eqref{lf} is quasiisomorphic to $\LL_{\langle\CO(g),\CU^\vee(g)\rangle}(i_*p^*\CO_S(2g-F))$.
Indeed, the first is clear since
$$
\Ext^\bullet(\CO(h-g),\CO_{\Sigma}(-1)) \cong
H^\bullet(\CO(g-h)\otimes\CO_{\Sigma}(-1)) \cong
H^\bullet(\Sigma,\CO) \cong \C.
$$
For the second it suffices to check that~\eqref{lf}, or equivalently~\eqref{rpi}, is orthogonal to $\CO(g)$ and $\CU^\vee(g)$.
For this we check that both terms of~\eqref{rpi} are orthogonal to $\CO(g)$ and $\CU^\vee(g)$.
Indeed, for the second term, $\CO_\Sigma(-1)$, we have
$$
\begin{array}{l}
\Ext^\bullet(\CO(g),\CO_{\Sigma}(-1)) \cong H^\bullet(\Sigma,\CO(-2)) = 0,\\
\Ext^\bullet(\CU^\vee(g),\CO_{\Sigma}(-1)) \cong H^\bullet(\Sigma,\CU(-2)) = 0.
\end{array}
$$
Further, for the first term, $\CO(h-g)$, we have
$$
\begin{array}{l}
\Ext^\bullet(\CO(g),\CO(h-g)) = H^\bullet(M,\CO(h-2g)) = H^\bullet(M,\CO(g-e)) = H^\bullet(\Gr(2,V),I_S(g)),\\
\Ext^\bullet(\CU^\vee(g),\CO(h-g)) = H^\bullet(M,\CU(h-2g)) = H^\bullet(M,\CU(g-e)) = H^\bullet(\Gr(2,V),I_S\otimes \CU(g)).
\end{array}
$$
Using resolution~\eqref{os} and Borel--Bott--Weil Theorem we conclude that these cohomology vanish.

Thus we have checked that
\begin{equation}\label{llll}
\LL_{\CO(h-g)}(\CO_{\Sigma}(-1)) \cong \LL_{\langle\CO(g),\CU^\vee(g)\rangle}(i_*p^*\CO_S(2g-F)).
\end{equation}
On the other hand, recall that by Lemma~\ref{mutfun} the mutation functors $\LL_{\CO(h-g)}$ and $\RR_{\CO(h-g)}$
are mutually inverse equivalences between ${}^\perp\CO(h-g)$ and $\CO(h-g)^\perp$.
So, let us check that $\CO_{\Sigma}(-1) \in {}^\perp\CO(h-g)$,
i.e. that $\Ext^\bullet(\CO_{\Sigma}(-1),\CO(h-g)) = 0$.
Indeed, by Serre duality we have
$$
\Ext^t(\CO_{\Sigma}(-1),\CO(h-g))^\vee =
\Ext^{4-t}(\CO(2h),\CO_{\Sigma}(-1)) = H^{4-t}(\Sigma,\CO_\Sigma(-1)) = 0.
$$
Thus~\eqref{llll} implies that
$$
\RR_{\CO(h-g)}\LL_{\langle\CO(g),\CU^\vee(g)\rangle}(i_*p^*\CO_S(2g-F)) \cong \CO_\Sigma(-1).
$$
Finally note that $\langle V/\CU, \CO(g) \rangle = \langle \CU^\perp(g), \CO(g) \rangle = \langle \CO(g), \CU^\vee(g) \rangle$,
hence $\LL_{\langle V/\CU, \CO(g) \rangle} = \LL_{\langle \CO(g), \CU^\vee(g) \rangle}$.
In view of this, the above isomorphism proves the Proposition.
\end{proof}

\subsection{The proof of the Main Theorem}

Now we are in a position to prove the Main Theorem.

Indeed, let us start with a semiorthogonal decomposition~\eqref{som10}.
Replacing $\D^b(S)$ by its semiorthogonal decomposition~\eqref{dbs}, twisting by $\CO_S(2g)$, and taking into account Proposition~\ref{FS} we obtain
the following semiorthogonal decomposition
$$
\D^b(M) = \langle \CS_{-1}(-g), \CS_0(-g), \CS_1(-g), \CS_2(-g), \CO(-h), \CO, \{ \CO_{\Sigma^+_i}(-1) \}_{i=1}^{10}, \Psi_2(\CA_S\otimes\CO_S(2g)) \rangle.
$$
Let us twist it by $\CO_M(g)$. As a result we obtain
\begin{equation}\label{som11}
\D^b(M) = \langle \CS_{-1}, \CS_0, \CS_1, \CS_2, \CO(g-h), \CO(g), \{ \CO_{\Sigma^+_i} \}_{i=1}^{10}, \Psi_3(\CA_S) \rangle,
\end{equation}
where
$\Psi_3 = \TT_{\CO(g)}\circ\Psi_2\circ\TT_{\CO(2g)}$.
It remains to make several additional mutations.

Step 9.
Apply the left mutation through $\CO(g-h)$ and $\CO(g)$ to the block $\{ \CO_{\Sigma_i^+} \}_{i=1}^{10}$.
The sheaves $\CO_{\Sigma^+_i}$ are completely orthogonal to $\CO(g)$ and $\CO(g-h)$, since
$$
\Ext^\bullet(\CO(g-th),\CO_{\Sigma^+_i}) = H^\bullet(\Sigma_i^+,\CO_{\Sigma_i^+}(-1)) = 0.
$$
Therefore, by Lemma~\ref{perpmut} we obtain
\begin{equation}\label{som12}
\D^b(M) = \langle \CS_{-1}, \CS_0, \CS_1, \CS_2, \{ \CO_{\Sigma^+_i} \}_{i=1}^{10}, \CO(g-h), \CO(g), \Psi_3(\CA_S) \rangle.
\end{equation}

Step 10.
Translate $\CO(g-h)$, $\CO(g)$ and $\Psi_3(\CA_S)$ to the left. By Lemma~\ref{longmut} we obtain
\begin{equation}\label{som13}
\D^b(M) = \langle \CO(-2h), \CO(-h), \Psi_4(\CA_S), \CS_{-1}, \CS_0, \CS_1, \CS_2, \{ \CO_{\Sigma^+_i} \}_{i=1}^{10} \rangle,
\end{equation}
where
$\Psi_4 = \TT(-g-h)\circ\Psi_3 = \TT_{\CO(-h)}\circ\Psi_2\circ\TT_{\CO(2g)}$.

Comparing this decomposition with Theorem~\ref{dbm} we conclude that
$$
\Phi_0(\D^b(X_+)) = \langle \CO_M(-2h), \CO_M(-h), \Psi_4(\CA_S) \rangle.
$$
By Lemma~\ref{phi0} the functor $\Phi_0$ commutes with the twist by $\CO(-h)$, hence we can rewrite
$$
\Phi_0(\D^b(X_+)) = \langle \CO_M(-h), \CO_M, \Psi_5(\CA_S) \rangle,
$$
where
$\Psi_5 = \TT_{\CO(h)}\circ\Psi_4 = \Psi_2\circ\TT_{\CO(2g)}$.
Since again by Lemma~\ref{phi0} we have $\CO_M(-th) = \Phi_0(\CO_{X^+}(-th))$,
and $\Phi_0$ is fully faithful, we conclude that
\begin{equation}
\Psi_5(\CA_S) = \Phi_0\circ\TT_{\CO(2h)}(\CA_{X^+}).
\end{equation}

\end{document}